\documentclass[12pt]{amsart}
\usepackage{amsmath, amssymb, amsfonts, graphics, xcolor, amsthm, dsfont, mathabx, mathrsfs}

\usepackage{soul}
\usepackage{marginnote}
\usepackage{csquotes}
\usepackage[pdftex]{graphicx}
\usepackage{epsfig}
\usepackage{todonotes}

\usepackage{latexsym}
\usepackage{euscript}
\usepackage{tocvsec2}
\usepackage{etoolbox}
\usepackage[font=small, margin=1em]{caption}
\usepackage{float}
\usepackage{tikz}
\usetikzlibrary{snakes}
\usetikzlibrary{decorations.pathreplacing}
\usepackage{enumitem}
\usepackage{url}

\usepackage[most]{tcolorbox}

\theoremstyle{plain}
\newtheorem{theorem}{Theorem}[section]
\newtheorem*{theorem*}{Theorem}
\newtheorem{lemma}[theorem]{Lemma}
\newtheorem{proposition}[theorem]{Proposition}
\newtheorem{corollary}[theorem]{Corollary}

\newtheorem{remark}[theorem]{Remark}

\newtheorem{problem}[theorem]{Problem}

\theoremstyle{definition}
\newtheorem{definition}{Definition}[section]
\newtheorem{algorithm}[definition]{Algorithm}

\def\F{\mathcal F}

\def\C{\mathcal C}

\def\X{\mathcal X}
\def\calF{\mathcal{F}}
\def\ft{\mathrm{ft}}
\def\pw{\mathrm{pw}}

\def\bb1{\mathds 1}

\renewcommand{\subset}{\subseteq}
\newcommand{\rev}[1]{\overleftarrow{#1}}
\newcommand{\sources}{\mathsf{Sources}}
\newcommand{\sinks}{\mathsf{Sinks}}
\newcommand{\white}{\mathsf{White}}
\newcommand{\fas}{\mathsf{FAS}}
\newcommand{\sm}{\setminus}

\newcommand{\ie}{{i.e.}}

\usepackage{float}
\usepackage{mdframed}
\mdfdefinestyle{MyFrame}{
    linecolor=black,
    outerlinewidth=2pt,
    innertopmargin=4pt,
    innerbottommargin=4pt,
    innerrightmargin=6pt,
    innerleftmargin=6pt,
    leftmargin = 4pt,
    rightmargin = 4pt
}

\begin{document}

\title{An approximation algorithm for zero forcing}

\author[B. Cameron]{Ben Cameron}
\address{School of Mathematical and Computational Sciences, University of Prince Edward Island, Charlottetown, PE, Canada.}
\email{brcameron@upei.ca}

\author[J. Janssen]{Jeannette Janssen}
\address{Department of Mathematics \& Statistics, 
Dalhousie University, Halifax, Nova Scotia, Canada.}
\email{jeannette.janssen@dal.ca}

\author[R. Mathew]{Rogers Mathew}
\address{Department of Computer Science \& Engineering, Indian Institute of Technology, Hyderabad}
\email{rogers@cse.iith.ac.in}

\author[Z. Zhang]{Zhiyuan Zhang}
\address{Toronto Metropolitan University}
\email{zhiyuan.zhang@torontomu.ca}


\begin{abstract}
We give an algorithm that finds a zero forcing set which approximates the optimal size by a factor of $\pw(G)+1$, where $\pw(G)$ is the pathwidth of $G$. Starting from a path decomposition, the algorithm runs in $O(nm)$ time, where $n$ and $m$ are the order and size of the graph, respectively. As a corollary, we obtain a new upper bound on the zero forcing number in terms of the fort number and the pathwidth. The algorithm is based on a correspondence between zero forcing sets and forcing arc sets. 
This correspondence leads to a new bound on the zero forcing number in terms of vertex cuts, and to new, short proofs for known bounds on the zero forcing number.
\end{abstract}

\maketitle

\section{Introduction}

The \emph{zero forcing process} was first introduced in~\cite{PhysicsFirst} with a physical motivation and independently as a process in graph theory in~\cite{barioli2008zero}. The graph process is defined as follows.
Let $G = (V, E)$ be a finite simple graph (as all will be in this work unless indicated otherwise). Partition the vertex set into two sets $W$ and $B\neq\emptyset$. We say that the vertices in $B$ are coloured blue, and those in $W$ are coloured white. The \textit{colour change rule} that we will use throughout this paper is that for a vertex $v\in W$, this vertex will change to blue if and only if there is a vertex, $u\in B$, adjacent to $v$ such that $N[u]\cap W=\{v\}$. That is, $v$ is the only white neighbour of $u$. In this case, we say that $u$ \emph{forces} $v$. The \textit{derived colouring} of a vertex subset $S\subseteq V$ is the colouring of the graph obtained by starting with only $S$ coloured blue and iteratively applying the colour change rule until no more vertices can change colour. 
The derived colouring of a vertex set is well-defined in the sense that, independent from the order of forcing, the derived colouring is unique~\cite{barioli2008zero,barioli2010problems}.
A \textit{zero forcing set} for a graph is a subset of the vertices of $G$ such that all vertices are blue in its derived colouring. The \textit{zero forcing number} of a graph, denoted $Z(G)$, is the minimum size of a zero forcing set for $G$.

The physical motivation of the zero forcing problem is as a scheme for controlling quantum systems~\cite{PhysicsQuantum}, and the graph-theoretic motivation comes from the relationship to the minimum rank of symmetric matrices associated with a given graph; see, for example, \cite{barioli2008zero,FALLAT2007558,hogben2022book}. 

\subsection{Related work}
There is a considerable amount of work on zero forcing and its many extensions and variations; see the recent book~\cite{hogben2022book} for an overview. Of special relevance for our work is the concept of \emph{forts}. A fort is a set of vertices such that no vertex outside the fort has exactly one neighbour in the fort. Since no vertex inside the fort can be forced by a vertex outside, a zero forcing set must contain at least one vertex of each fort. The maximum number of disjoint forts, known as the \emph{fort number}, thus provides a lower bound on the zero forcing number.
Forts were first defined in~\cite{fast2018effects} and have since been applied to zero forcing in many ways, including formulating integer programming models for zero forcing~\cite{BrimkovFastHicks2019}, bounding evaluations of the zero forcing polynomial~\cite{boyer2019zfpolyn}, considering the zero forcing process on various product and grid graphs~\cite{BEAUDOUINLAFON202035,KARST2020380}, and analyzing variations of the zero forcing process \cite{cameron2023forts}.

It turns out that computing the zero forcing number is NP-hard~\cite{aazami2008hardness}. 
Code that computes the zero forcing number is available at~\cite{sagecode}. 
Efficient algorithms for special families of graphs have been proposed, such as unicyclic graphs~\cite{row2012technique} and claw-free graphs~\cite{brimkov2022computer}.

Given the NP-hardness of the problem, the main focus of many works is on finding an upper bound on the zero forcing number in terms of graph parameters and structures, such as 
vertex cover number \cite{brimkov2022computer}, 
spectral radius \cite{das2023spectral, zhang2022zero}, 
girth and degree \cite{butler2013throttling, davila2018bounds, GentnerPensoRautenbachSouza2016, liang2022extremal}, 
matchings\cite{davila2020matching, jing2023zero, trefois2013zero},
cyclomatic number \cite{jing2023zero, wang2020zero}, 
and path covers \cite{davila2020matching, GentnerPensoRautenbachSouza2016, montazeri2020relationship}.  
The zero forcing number for specific families of graphs has also been extensively considered; for example, 
Johnson graphs and generalized Grassmann graphs \cite{abiad2023diameter}, generalized Petersen graphs \cite{alameda2018families},
circulant graphs \cite{duong2020maximum},
unicyclic graphs \cite{row2012technique},
claw-free graphs \cite{brimkov2022computer},
random graphs \cite{bal2021zero, Kalinowskietal2019}, and 
2-connected planar graphs \cite{ison2023zero}.

\subsection{Forcing arc sets}
It was observed early on \cite{barioli2010problems} that 
an instance of the zero forcing process on a graph $G=(V, E)$ with an initial blue set $B_0$ can be characterized by a set of \emph{forcing chains}. These are directed paths $v_1v_2\dots v_k$ so that $v_1\in B_0$, and for all $1\leq i<k$, $v_i$ \emph{forces} $v_{i+1}$ (that is, changes $v_{i+1}$ from white to blue). Since no vertex can force more than one other vertex,  the forcing chains are vertex-disjoint.

It can be deduced from the forcing process that any forcing chain must be an \emph{induced} path. This observation has an immediate corollary that $Z(G)\ge P(G)$ where $P(G)$ is the \textit{path cover number}, that is, the minimum number of induced paths such that every vertex of $G$ lies on exactly one of these paths. In~\cite{GentnerPensoRautenbachSouza2016}, it was shown that if $Z(H)=P(H)$ for all induced subgraphs $H$ of $G$, then $G$ is a cactus. 

In our paper, we define \emph{forcing arc sets}, which are the sets of arcs that make up a set of forcing chains. That is, each forcing arc $(u,v)$ represents the relation ``$u$ forces $v$''. 
Forcing arc sets were independently introduced as so-called parallel increasing path covers (or PIPs) in~\cite{PIPs}. The concepts are not entirely equal: PIPs are represented explicitly as collections of directed paths while forcing arc sets are collections of arcs which implicitly define the paths.

\subsection{Main results}

Our main result is an approximation algorithm for finding a zero forcing set of a graph. This algorithm is stated as Algorithm~\ref{algorithm} and its correctness is shown in Theorem \ref{thm: proof of correctness}. The algorithm assumes that a \emph{path decomposition} of the graph is given, and the approximation ratio of the algorithm depends on the width of the decomposition. 
If the decomposition is optimal, then for a graph $G$ with pathwidth $\pw(G)$, the algorithm returns a zero forcing set $S$ so that $|S|\leq (\pw(G)+1)Z(G)$. Our algorithm runs in $O(nm)$, where $n $ and $m$ are the order and size of the graph, respectively.

Our algorithm returns a zero forcing set and a collection of disjoint forts, which provide a lower bound on the zero forcing number and, therefore, on the approximation ratio. The algorithm thus also leads to a new bound on the ratio between the zero forcing number and the fort number. This bound is stated as Corollary \ref{cor:bound}. Specifically, we show that the ratio between the fort number and the zero forcing number is bounded above by $\pw(G)+1$. 
While it was known that the {\sl fractional} versions of zero forcing number and fort number are equal~\cite[Remark 2.6]{cameron2023forts}, previously no upper bound on the gap between the fort number and zero forcing number was known. 

The algorithm is based on  a characterization of forcing arc sets in terms of certain types of cycles formed by arcs in the forcing arc set and edges of $G$ which we call \emph{chain twists}. Chain twists are formally defined in Definition \ref{def:chain twist}. Loosely stated, we show that a set of arcs is a forcing arc set if and only if it does not contain a chain twist. This result is given as Theorem \ref{thm:chaintwist}. A similar result is shown in \cite[Theorem 3.8]{PIPs}: they define a \emph{relaxed chronology} which essentially is a set of directed paths without a chain twist, and they then show that a set of paths is a PIP if and only if it is a relaxed chronology. However, our representation in terms of arcs makes it easy to combine forcing arc sets from different parts of a graph.

In particular, if a graph $G=(V,E)$ has a vertex cut $C$ and $V_1$ is the vertex set of a connected component of $G-C$, then under certain conditions we can combine forcing arc sets in $G_1=G[V_1\cup C]$ and in $G_2=G[V\setminus V_1]$ to find a forcing arc set in $G$. This result is given in Proposition \ref{prop:cut}. This combination of arc sets is an essential ingredient of the algorithm, allowing us to use the fact that each set of the path decomposition of a graph is a vertex cut to iteratively build a forcing arc set for the whole graph. The proposition also directly leads to a bound on $Z (G)$ in terms of the zero forcing number of its constituent parts. This bound is stated in Proposition \ref{prop:cutbound}.

There is a simple relationship between the size of a zero forcing set and the number of forcing arcs representing a forcing process starting from this set. This insight allows us to give new, short proofs of existing results. 
In particular, we give new proofs of a known result on the zero forcing number of proper interval graphs and of an upper bound on the zero forcing number of the strong product of two graphs in terms of the zero forcing number of the two components. We also give a new result on the zero forcing number of a graph with a vertex cut in terms of the zero forcing number of the constituent subgraphs. This is  Proposition~\ref{prop:cut}. The proofs follow immediately from the basic properties of forcing arc sets and chain twists. Moreover, the proofs are \emph{constructive}, in the sense that they describe how to obtain an optimal forcing arc set.

\subsection{Pathwidth}

Our algorithm requires that we first obtain a path decomposition. Determining whether the pathwidth of a graph $G$ is at most $k$ is NP-complete \cite{Arnborg}. The pathwidth problem is known to be NP-complete for many other classes of graphs like bounded degree graphs, planar graphs, and chordal graphs; see \cite{Bodlaender_tourist} for more details). However, there are {\em fixed-parameter tractable} (FPT) algorithms that run in $O(f(k)\cdot n)$, where $f(k)$ is exponential in $k$, 
to determine if the pathwidth of a given graph on $n$ vertices is at most $k$ and if so give a path decomposition of width $k$ (for example, \cite{BODLAENDER1}). Thus, for graphs with constant pathwidth, we can use the FPT algorithm mentioned above to obtain an optimal path decomposition in linear time and then give that as an input to our algorithm to find a zero forcing set $S$. Our results guarantee that the size of $S$ is at most a constant times the size of an optimal zero forcing set.

Our algorithm thus gives an efficient way to approximate the zero forcing number for graphs with low pathwidth.
Regarding graphs of low pathwidth, it is known (see~\cite{Bodlaender2012}) that for any given forest $H$, there is a constant $c_H$ such that every graph that does not have $H$ as a minor has its pathwidth bounded from above by $c_H$. Very recently, it was shown in \cite{BrianskiJMMSS23} that the pathwidth of any $2$-connected graph with circumference $k$ is at most $k-1$, where \emph{circumference} of a graph with at least one cycle is the length of the longest cycle in it. It was shown by Korach and Solel~\cite{KORACH_pw} that the pathwidth of a graph on $n$ vertices is at most $c\log n$ times its treewidth, where $c$ is a constant. Thus, graphs with low or constant treewidth, like series-parallel graphs and Halin, have low pathwidth.

\subsection{Organization}

The remainder of the paper is organized as follows. In Section~\ref{sec:chaintwists} we 
give the definition of forcing arc sets and chain twists, and prove the theorem which relates forcing arc sets to chain twists.
In Section~\ref{sec:vertexcut} we give our results on forcing arc sets and zero forcing sets in graphs with a vertex cut. 
In Section~\ref{sec:approxalg} we give the approximation algorithm, prove its correctness, and derive its complexity. 
Finally, in Section~\ref{sec:newproofs}, we derive simple facts about forcing arc sets which lead to new proofs for known results.

\section{Forcing arc sets and chain twists}\label{sec:chaintwists}

As stated earlier, a forcing arc set is a set of directed edges, or arcs, that represent a zero forcing process. Specifically, an arc $(u,v)$ in the forcing arc set represents the relation ``$u$ forces $v$'' 
The question we answer in this section is: when is an arc set a forcing arc set? We will answer this question by replacing the characterization of forcing arcs in terms of forcing chains with a structural condition that is independent of any zero forcing process. To do this, we consider conditions that prevent a set of arcs from being a forcing arc set.

Before we can formalize the concept of forcing arc sets, we need to address the fact that the relation ``$u$ forces $v$'' can be ambiguous. It may be that, during the zero forcing process, two blue vertices share a unique white neighbour $w$. We will assume that each white vertex is forced by one neighbour only. This implies that the same zero forcing process may be represented by different forcing arc chains. 

A \emph{source} in the digraph is a vertex of in-degree zero, while a \emph{sink} is a vertex with out-degree zero.
In this work, for a graph $G = (V, E)$, an {\em arc set $A\subseteq V\times V$ of $G$} is a collection of arcs $(u,v)$ so that $uv\in E$. Arc sets are required to be unidirectional so that $(v, u)\not\in A$ if $(u, v)\in A$. 
We will say $v$ is a sink ({\em resp.} source) of $A$ if $v$ is a sink ({\em resp.} source) of the directed graph $H = (V, A)$ if the context is clear.
\begin{definition}\label{def:fas} 
Given a graph $G = (V, E)$, an arc set $A$ of $G$ is a \emph{forcing arc set} if 
\begin{enumerate}[label=(P\arabic*),align=right]
\item\label{cond:fas1} the directed graph $H=(V, A)$ consists of a collection of vertex-disjoint dipaths, and
\item\label{cond:fas2} 
there is a zero forcing process on $G$ and $(u,v)\in A$ precisely if blue vertex $u$ forces white vertex $v$ during a step of the process. 
\end{enumerate}  
We say that $A$ \emph{represents} the zero forcing process.
The set of sources of $A$ is the zero forcing set of the process. 
\end{definition}

By our convention that a vertex can only be forced by one other vertex in a zero forcing process, we have that each vertex in $V$ can have at most one in-neighbour. From the definition of the zero forcing process, we also have that each vertex can have at most one out-neighbour and that vertices in the zero forcing set have in-degree zero.  Thus,   \ref{cond:fas2} implies \ref{cond:fas1}; we included \ref{cond:fas1} to make the condition explicit. The dipaths in $H=(V,A)$ are exactly the forcing chains of the zero forcing process. Note that such a path can have length zero, corresponding to a vertex in the zero forcing set that did not force any other vertex.

We will now replace \ref{cond:fas2} in Definition~\ref{def:fas} with a structural condition that is independent of any zero forcing process. To do this, we consider conditions that prevent a set of arcs from being a forcing arc set. 

Firstly, each forcing chain must correspond to an \emph{induced path} of $G$, which was first observed in \cite[Proposition 2.10]{barioli2010problems}. If $v_0v_1\dots v_{k-1}$ is a forcing chain, then $G$ cannot contain an edge $v_iv_j$ with $j>i+1$ ($v_i$ cannot force $v_{i+1}$ while the vertex $v_j$ is still white since it would have at least two white neighbours). 

To see a second obstruction, consider Figure \ref{fig:chaintwist} that portrays three dipaths. 
We see that the three forcing chains will stop at $u_i$, $v_{i'}$, and $w_{i''}$, respectively as each requires another of its neighbours higher on another forcing chain to turn blue before it can force further, and thus the zero forcing process halts.  Thus, a configuration as in Figure~\ref{fig:chaintwist} can not occur in a forcing arc set. This corresponds to a set of directed paths connected by edges of $G$ that are not part of the arc set, which leads to the following definition.

\begin{figure}
	\centering
		\begin{tikzpicture}[shorten >=1pt,->]
 		\tikzstyle{vertex}=[circle,draw=black!60,minimum size=12pt,inner sep=2pt]
        \tikzstyle{path} = [->, dashed, line width = 1.3pt]
 		\node[vertex] (v11) at (0,0) {$u_{1}$};
 		\node[vertex] (u1) at (2,0) {$u_{i}$};
 		\node[vertex] (w1) at (4,0) {$u_{j}$};
 		\node[vertex] (v14) at (6,0) {$u_{l}$};
  
  		\node[vertex] (v21) at (0,-1.5) {$v_{1}$};
  		\node[vertex] (u2) at (2,-1.5) {$v_{i'}$};
  		\node[vertex] (w2) at (4,-1.5) {$v_{j'}$};
 		\node[vertex] (v24) at (6,-1.5) {$v_{m}$};
  
  		\node[vertex] (v31) at (0,-3) {$w_{1}$};
  		\node[vertex] (u3) at (2,-3) {$w_{i''}$};
  		\node[vertex] (w3) at (4,-3) {$w_{j''}$};
  		\node[vertex] (v34) at (6,-3) {$w_{k}$};

  		\draw[path] (v11) -- (1.7,0);
        \draw[path] (u1) -- (3.65,0);
        \draw[path] (w1) -- (5.7,0);
        
        \draw[path] (v21) -- (1.65,-1.5);
        \draw[path] (u2) -- (3.65,-1.5);
        \draw[path] (w2) -- (5.65,-1.5);
        
  	    \draw[path] (v31) -- (1.6,-3);
        \draw[path] (u3) -- (3.6,-3);
        \draw[path] (w3) -- (5.65,-3);     
  
  		\draw (u1)--(w2)--cycle; 
  		\draw (u2)--(w3)--cycle;
  		\draw (u3)--(w1)--cycle; 
		\end{tikzpicture}
		\caption{A set of dipaths that contains a chain twist}
		\label{fig:chaintwist}  
\end{figure}
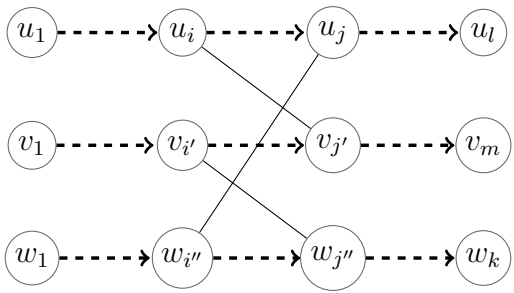%

\begin{definition}[Chain twist]\label{def:chain twist}
Given a graph $G$ and an arc set $A$ of $G$, a \emph{chain twist in $A$} is a cycle $v_0v_1\dots v_{k-1}v_0$ in $G$ so that, for all $0\leq i\le k$, 
\begin{equation}\label{eq:chaintwist}\tag{CT}
    (v_{i},v_{i+1})\not\in A\implies (v_{i-1},v_{i})\in A\text{ and }(v_{i+1},v_{i+2})\in A,
\end{equation}
where the addition of indices is taken modulo $k$.
If a cycle in $G$ satisfies \eqref{eq:chaintwist}, we say {\em $A$ contains a chain twist}.
\end{definition}
\begin{remark}\label{rmk:remark-ct}
Note that the occurrence of a chain twist with only arcs in $A$ implies that $H$ contains a directed cycle, so $A$ cannot be a collection of forcing chains. 
The occurrence of a chain twist with exactly one edge not in $A$ implies that one of the dipaths formed by arcs of $A$ does not correspond to an induced path; therefore this dipath cannot be a forcing chain.
\end{remark}

\begin{theorem}\label{thm:chaintwist}
Let $G=(V, E)$ be a graph and an arc set $A$ of $G$ that satisfies \ref{cond:fas1}. 
The arc set $A$ is a forcing arc set if and only if $A$ does not contain a chain twist. 
\end{theorem}

\begin{proof}
Suppose first that $A$ is a forcing arc set. Suppose by contradiction that $A$ contains a chain twist $v_0\dots v_{k-1}v_k$, where $v_k=v_0$.
Without loss of generality, we assume that $(v_{k-1},v_0)\in A$.
We form a sequence $\{i_s\}_{s=0}^\ell\subset \{1,2,\dots,k\}$, $\ell\leq k$, of indices of vertices in the chain twist as follows. Let $i_0=0$. 
Fix $s\geq 0$ and assume $i_s$ has been defined where $i_s\leq k-1$, we set $i_{s+1}$ to be 
$$i_{s+1}= \left\{\begin{array}{rl}
    i_s+1 & \text{if } (v_{i_s},v_{i_s+1})\in A; \\
    i_s+2 & \text{otherwise.}
\end{array}\right.$$
The process ends when $i=i_\ell=k$. Note that $\{i_s\}$ is increasing and $i_\ell=k$ is always reached since $(v_{k-1},v_k)\in A$ so if $i_s=k-1$, then $i_{s+1}=k$.

 Consider the zero forcing process represented by $A$ and assume one vertex turns from white to blue in each step of the process.
We now argue that, for all $0\leq s<\ell$, $v_{i_s}$ turns blue before $v_{i_{s+1}}$ turns blue. Fix $0\leq s<\ell$ and let $i_s=i$. If $(v_i,v_{i+1})\in A$, then $v_i$ forces $v_{i+1}$, so $v_i$ turns blue before $v_{i+1}$ turns blue. Since we set $i_{s+1}=i+1$, the statement holds.
If $(v_i,v_{i+1})\not\in A$, then $v_iv_{i+1}\in E$ and $(v_{i+1},v_{i+2})\in A$, so $v_{i+1}$ forces $v_{i+2}$. The vertex $v_{i+1}$ can only force $v_{i+2}$ once its other neighbour $v_i$ has turned blue. So $v_i$ turns blue before $v_{i+2}$ turns blue.
In this case, we set $i_{s+1}=i+2$, so the statement again holds. 

We have now shown that for all $0\leq s<\ell$, $v_{i_s}$ turns blue before $v_{i_{s+1}}$ turns blue. 
From this, we can conclude that $v_0$ turns blue before $v_{i_\ell}$ turns blue. Since $i_\ell=k$ and $v_k=v_0$, this leads to a contradiction.

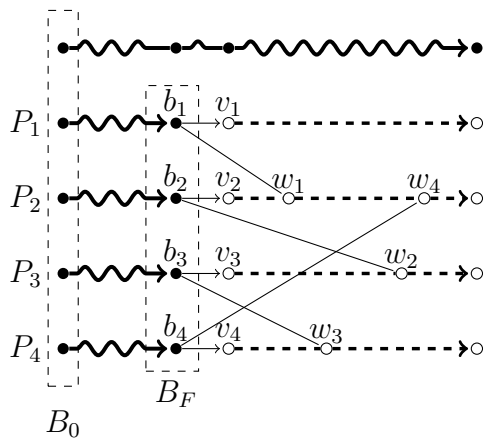
\begin{figure}
\centering
   
\begin{tikzpicture}[shorten >=1pt,->]
\tikzstyle{solid}=[circle, fill, inner sep=1.5pt]
\tikzstyle{empty}=[circle,draw=black, inner sep=1.5pt]
\tikzstyle{bold}=[->, snake=snake, line after snake=2mm, line before snake=1mm, line width=1.5pt]
\tikzstyle{dashed snake} = [->, dashed, line width = 1.3pt]
\tikzstyle{edge}=[-]
\tikzstyle{vertex}=[]

\node[solid] (00) at (0,0) {};
\node[solid] (01) at (1.5,0) {};
\node[solid] (03) at (2.2,0) {};
\node[solid] (09) at (5.5,0) {};
\draw[bold] (00) --  (01) -- (03) -- (09);

\foreach \i in {-1,...,-4}{ 
    \node[solid] (\i0) at (0,\i) {};
    \node[solid] (\i1) at (1.5,\i) {};
    \draw[bold] (\i0) -- (\i1);
}

\foreach \i in {1,...,4}{
    \node at (-.5, -\i) {$P_{\i}$};
}

\draw[black, dashed] (-0.2, 0.5) rectangle (0.2, -4.5); 
\node at (0, -5) {$B_0$};
\foreach \i in {-1,...,-4}{
    \node[empty] (\i3) at (2.2,\i) {};
    \node[empty] (\i9) at (5.5,\i) {}; 
    \draw (\i1) -- (\i3);
}
\foreach \i in {1,...,4}{
    \node at (1.5, -\i+.28) {$b_{\i}$};
}

\foreach \i in {1,...,4}{
    \node at (2.2, -\i+.22) {$v_{\i}$};
}


\node[empty] (u0) at (3, -2) {};
\node at (3, -1.8) {$w_1$};
\node[empty] (u1) at (4.5, -3) {};
\node at (4.5, -2.8) {$w_2$};
\node[empty] (u2) at (3.5, -4) {};
\node at (3.5, -3.8) {$w_3$};
\node[empty] (u4) at (4.8, -2) {};
\node at (4.8, -1.8) {$w_4$};

\draw[dashed snake] (-13) -- (-19);

\draw[edge] (-11) -- (u0);
\draw[edge] (-21) -- (u1);
\draw[edge] (-31) -- (u2);
\draw[edge] (-41) -- (u4);

\draw[dashed snake] (-23) -- (u0) -- (u4) -- (-29); 
\draw[dashed snake] (-33) -- (u1) -- (-39); 
\draw[dashed snake] (-43) -- (u2) -- (-49); 
\draw[dashed] (1.1, -.5)rectangle(1.8, -4.3) ; 
\node at (1.5, -4.6) {$B_F$};

\end{tikzpicture}
\caption{An illustration for the proof of the converse direction of Theorem~\ref{thm:chaintwist}.
The figure consists of $H = (V, A)$ that satisfies \ref{cond:fas1}, and the remaining edges in $G$ are omitted. 
Filled vertices correspond to blue vertices and unfilled vertices correspond to white vertices. 
Each dipath of $H$ that is not all blue consists of an initial segment of blue vertices, indicated by a squiggly arrow, followed by an arc $(b_i,v_i)$, indicated by a thin arrow, followed by a final segment of white vertices, indicated by a dashed arrow. 
Thin lines represent edges that are not in $A$.
In this case, we form a digraph $D$ with a vertex set $\{1,2,3,4\}$ and an edge set $\{(1,2), (2,3), (3,4), (4,2)\}$ where there is one directed cycle on $2,3,4$. 
}
\label{fig:enlighten}
\end{figure}

We will prove the converse by contrapositive, so we suppose that $A$ is not a forcing arc set and we will show that it contains a chain twist. 

Let $B_0$ be the set of sources of $A$; this includes the set of isolated vertices in $H = (V, A)$. We will now run the zero forcing process on $G$ starting from $B_0$ as follows: in each step, we choose an arc $(u,v)\in A$ so that $u$ is blue and $v$ is the only white neighbour of $u$, and we colour $v$ blue. 
We repeat this until no such arcs exist; let $B\subset V$ denote the set of blue vertices. 

By our assumption, $A$ is not a forcing arc set. 
Therefore, this process stops when there are still white vertices. 
By \ref{cond:fas1}, $H=(V, A)$ consists of a set of vertex-disjoint dipaths.  
Let $P_1, \dots, P_k$ denote the dipaths that still contain a white vertex. 
By definition, $B_0\subseteq B$, and thus each $P_i$ starts with consecutive blue vertices, which include the source, followed by consecutive white vertices, which include the sink. 
This implies that all dipaths $P_i$ contain an arc with a blue tail and a white head.

For $1\leq i\leq k$, we define three special vertices $b_i$, $v_i$ and $w_i$. Let $b_i$ denote the blue vertex in $P_i$ such that the next vertex $v_i$ (in the sense that $(b_i, v_i)\in A$) is white; therefore, $b_i$ and $v_i$ lie on $P_i$. 
Since we have completed all possible forcing on $G$, each $b_i$ has at least two white neighbours, where $v_i$ is one of them; 
let $w_i$ be another white neighbour of $b_i$. 
Note that if there exists $1\leq i\leq k$ such that $w_i$ and $b_i$ both lie on the same path $P_i$, then it forms a chain twist (see Remark~\ref{rmk:remark-ct}) and we are done; therefore, we assume that $w_i$ does not lie on $P_i$. 
Since $A$ consists of vertex-disjoint dipaths, for every $i$, $1\leq i\leq k$, there is a unique $j$, $1\leq j\leq k$ and $i\neq j$, so that $w_i$ lies on $P_j$. For an illustration, see Figure \ref{fig:enlighten}.

We can now find a chain twist through an auxiliary digraph {whose vertex set consists of the indices of the chain twist}. Form a digraph $D$ with vertex set $\{1,2,\dots, k\}$ which has an arc $(i,j)$ precisely when $w_i$ lies on path $P_j$. By definition, every vertex has out-degree one in $D$. Therefore, $D$ contains a directed cycle $C=i_0i_1i_2\dots i_\ell$, where $i_\ell=i_0$ and $2\leq \ell\leq k$ is the length of the cycle. Since we can always relabel the paths, we assume that $i_j=j+1$, $0\leq j<\ell$, so the cycle has vertices $1, 2,\dots, \ell$. 
Specifically, each white vertex $w_i$ lies on $P_{i+1}$ for $1\leq i\leq \ell$, and $w_\ell$ lies on $P_1$.

We introduce a notation that will be helpful: let $P$ be a dipath and let $u$ and $v$ be two vertices of $P$, where $u$ is comes earlier on the path than $v$, let $P[u:v]$ denote the section of the dipath following the arcs from $u$ to $v$. 
Now, consider the closed walk 
$$ W= P_{\ell}[b_{\ell}: w_{\ell-1}]P_{\ell-1}[b_{\ell-1}: w_{\ell-2}]\dots P_{\ell'-1}[b_{2}:w_{1}]P_{1}[b_{1}:w_{\ell}]b_{\ell}. $$
Note that, by definition, for all $1\leq j\leq \ell$, $b_jw_j$ is an edge in $G$ that is not part of $A$ and $P[b_j:w_{j-1}]$ (consider $u_0=u_k$) is a segment of path $P_j$. Thus
$W$ is indeed a closed walk in $G$. Since $C$ is a cycle, all path segments are distinct. Since $P_1, \dots, P_k$ is a collection of vertex-disjoint dipaths, the closed walk $W$ is a cycle in $G$. Since all $b_i$ are blue and $w_i$ white, each path segment is a dipath in $H$ containing at least one arc. So $W$ consists of a sequence of non-empty dipaths with arcs in $A$ connected by single edges in $E$. Therefore, $W$ is a chain twist. 
\end{proof}

With this theorem in hand, we can also define a forcing arc set of graph $G=(V, E)$ to be an arc set $A$ of $G$ that satisfies the following two conditions:
\begin{enumerate}[label=(Q\arabic*),start=2] {\sl 
\item[\ref{cond:fas1}] the directed graph $H=(V, A)$ consists of a collection of vertex-disjoint dipaths, and
\item\label{cond:fas3} $A$ does not contain a chain twist.}
\end{enumerate}

The following lemma follows directly from the definition.
\begin{lemma}\label{lem:subset}
If $A$ is a forcing arc set for $G$, then any subset $A'\subseteq A$ is a forcing arc set for $G$.
\end{lemma}

We now present a lemma about chain twists that will be of use in later sections. 

\begin{lemma}\label{lem:closedwalk}
Let $G$ be a graph and $A$ be an arc set of $G$.
If $G$ contains a closed walk $W=v_0v_1\dots v_{k-1}v_0 $ that satisfies \eqref{eq:chaintwist}, then $A$ contains a chain twist. 
\end{lemma}
\begin{proof}
We will show that every closed walk that satisfies \eqref{eq:chaintwist} is either a cycle (and thus a chain twist) or it contains a shorter closed walk satisfying \eqref{eq:chaintwist}. This proves the lemma since no minimal counterexample to the statement is possible.

Let $W=v_0v_1\dots v_{k-1}v_0 $ be a closed walk satisfying \eqref{eq:chaintwist}. 
Suppose $W$ is not a cycle; there then exist indices $s<t $ so that $v_s=v_t$. Let 
\[
W_1=v_sv_{s+1}\dots v_t\text{ and }W_2=v_tv_{t+1}\dots v_{k-1}v_0\dots v_{s-1}v_s
\]
be the two closed walks along $W$ from $v_s$ to $v_t$ and from $v_t$ back to $v_s$. 
Since $W_1$ belongs to $W$, all edges of $W_1$ satisfy \eqref{eq:chaintwist} except that possibly the consecutive edges $(v_{t-1},v_t)$ and $ (v_s,v_{s+1})$ are both not in $A$. 
But in that case, $(v_t,v_{t+1})\in A$ and $(v_{s-1},v_s)\in A$, which implies that $W_2$ satisfies \eqref{eq:chaintwist} with a similar argument. 
\end{proof}

\section{Zero forcing with vertex cuts}\label{sec:vertexcut}

In this section, we examine how the zero forcing number of a graph with a vertex cut relates to the zero forcing numbers of the graphs formed from the components created by removing the cut. We use the dichotomy between forcing arc sets and zero forcing sets. We start by giving a straightforward relation between the number of arcs in a forcing arc set and the size of the corresponding zero forcing set. This result will prove useful in what follows. 

\begin{proposition}\label{prop:sizeBA}
Given a graph $G=(V,E)$ and a zero forcing set $B_0\subset V$ for $G$. Let $A$ be a forcing arc set of $G$ representing a zero forcing process on $G$ with the initial blue set $B_0$, we have that 
\begin{equation*}
|A|=|V|-|B_0|.
\end{equation*}
\end{proposition}
\begin{proof}
Each component of $H = (V, A)$ is a dipath and, thus, contains exactly one vertex of $B_0$. Thus, $|B_0|$ equals the number of components of $H$. Since the underlying graph of $H$ is a forest, the number of components equals $|V|-|A|$.
\end{proof}

This immediately gives the following corollary. 
\begin{corollary}\label{cor:sizeBA}
    For any forcing arc set $A$ of a graph $G$ of order $n$, we have that 
    \[
    |A|\leq n-Z(G).
    \]\hfill\qed
\end{corollary}
We will refer to forcing arc sets that attain the bound from Corollary~\ref{cor:sizeBA} as \emph{optimal} forcing arc sets. 

Define the {\em reverse of an arc set $A$}, denoted by $\overleftarrow{A}$, as the arc set obtained by reversing all arcs in $A$. Formally, 
\begin{equation*}
    \overleftarrow{A}=\{ (v,u)\,:\,(u,v)\in A\}.
\end{equation*}
It follows from the definition that the reversal of a chain twist is, again, a chain twist. This immediately leads to the following proposition, also observed in \cite{hogben2012propagation}.  

\begin{proposition}\label{prop:reverse}
    An arc set $A$ of $G$ is a forcing arc set if and only if $\overleftarrow{A}$ is a forcing arc set. 
    \hfill \qed
\end{proposition}

The sources of $\overleftarrow{A}$ are the sinks of $A$. The sinks of a forcing arc set are the endpoints of the forcing chains. 
This leads to the following corollary that was first observed in \cite{barioli2010problems}.
\begin{corollary}
Let $B_0$ be a zero forcing set for $G$. Consider a set of forcing chains representing the zero forcing process with an initial blue set $B_0$. 
The endpoints of all forcing chains are again a zero forcing set. \hfill \qed
\end{corollary}

Obviously, a subset of a forcing arc set is again a forcing arc set. 
Forcing arc sets can also be combined to form new forcing arc sets. This is expressed in the following proposition.

\begin{proposition}\label{prop:cut}
    Let $G=(V,E)$ be a graph with a vertex cut $C$. Let $V_1$ be the vertex set of one of the components of $G-C$, and let $V_2=V\setminus (V_1\cup C)$. Let $G_1=G[V_1\cup C]$ and $G_2=G[V_2\cup C]$. If $A_1$ and $A_2$ are forcing arc sets in $G_1$ and $G_2$, respectively, so that each vertex in $C$ is a source of both $A_1$ and $A_2$, then $A_1\cup \overleftarrow{A_2}$ is a forcing arc set of $G$.   
\end{proposition}
\begin{proof}
Since every vertex in $C$ is a source in $A_2$, every vertex in $C$ is a sink in $\overleftarrow{A_2}$. By Proposition~\ref{prop:reverse}, $\overleftarrow{A_2} $ is a forcing arc set. Now consider $A=A_1\cup \overleftarrow{A_2}$. Each vertex in $C$ is a source in $A_1$ and a sink in $\overleftarrow{A_2}$, so it has out-degree and in-degree both at most one in $A$. 
Since $C$ is a vertex cut, the in-degree and out-degree of any vertex of $V_1$ ({\em resp. }$V_2$) in $A$ is as it was in $A_1$ ({\em resp.} $A_2$).
Thus, \ref{cond:fas1} is satisfied. Now, by Theorem \ref{thm:chaintwist}, we need to show only that $A$ does not contain a chain twist \ref{cond:fas3}.

Suppose $A$ contains a chain twist $v_0v_1\dots v_{k-1}v_0$.
Note that a chain twist corresponds to a cycle in $G$. 
Since $A_1$ and $A_2$ are forcing arc sets,  no chain twist can consist only of vertices from $G_1$, or only from $G_2$.
Therefore, the corresponding cycle must contain vertices from both $G_1$ and $G_2$. Since $V_1$ and $V_2$ are separated by $C$ and the cycle must cross from $G_1$ to $G_2$ and back, the chain twist must contain at least two vertices of $C$. 

Suppose without loss of generality that the chain twist is indexed such that $v_0\in C$ and $v_1\in V_1$. Let $s>1$ be the least index so that $v_{s+1}\in V_2$; since $C$ separates $V_1$ from $V_2$ this implies that $v_s\in C$. 
%
We now argue that $(v_{s-1},v_s)\not \in A$. Assume the contrary, so $(v_{s-1},v_s)$ is in $A_1$ or in $\overleftarrow{A_2}$. By definition, $v_{s-1}\not\in V_2$. If $(v_{s-1},v_s)\in \overleftarrow{A_2}$, then $v_{s-1}v_s$ is part of $G_2=G[V_2\cup C]$, and thus $v_{s-1}\in C$. But then $v_{s-1}$ is in $C$ but not a sink in $\overleftarrow{A_2}$, which contradicts our assumption. If  $(v_{s-1},v_s)\in A_1$ then $v_s$ is in $C$ but not a source in $A_1$, which also contradicts our assumption. Therefore, $(v_{s-1},v_s)\not \in A$

By \eqref{eq:chaintwist}, we then have that $(v_s,v_{s+1})\in A$. 
Since $v_s\in C$, $v_{s+1}\in V_2$, and $v_sv_{s+1}$ is an edge of $G_2$, this implies that $(v_s,v_{s+1})\in \overleftarrow{A_2}$. Since $v_s\in C$, this contradicts the fact that each vertex in $C$ is a sink in $\overleftarrow{A_2}$. 
\end{proof}

\begin{corollary}
    Let $G=(V, E)$ be a graph with a vertex cut $C$. Let $V_1$ be the vertex set of one of the components of $G-C$, and let $V_2=V\setminus (V_1\cup C)$. Let $G_1=G[V_1\cup C]$ and $G_2=G[V_2\cup C]$. If there exist zero forcing sets $S_1$ in $G_1$ and $S_2$ in $G_2$ so that $C\subseteq S_1\cap S_2$, then there exists a zero forcing set for $G$ of size $|S_1|+|S_2|-|C|$.
\end{corollary}
\begin{proof}
Note first that $|V|=|V_1|+|V_2|+|C|$.
Let $A_1$ and $A_2$ be the forcing arc sets in $G_1$ and $G_2$ corresponding to the zero forcing process on $G_1$ and $G_2$ with the initial blue sets $S_1$ and $S_2$, respectively. By Proposition \ref{prop:sizeBA}, $|A_1|=|V_1\cup C|-|S_1|$ and  $|A_2|=|V_2\cup C|-|S_2|$. By Proposition \ref{prop:cut}, $A=A_1\cup \overleftarrow{A_2}$ is a forcing arc set of $G$. The set of sources of $A$ is a zero forcing set for $G$ of size $|V|-|A|$, where
\begin{equation*}\begin{split}
    |V|-|A| & = |V_1\cup C| + |V_2\cup C|- |C| - (|A_1|+|A_2|)\\
    & =|S_1|+|S_2|-|C|.
\end{split}\end{equation*}
\end{proof}

Proposition~\ref{prop:cut} also gives a bound on the zero forcing number derived from a vertex cut.
The special case of graphs containing a cut vertex (\ie~$|C|=1$ in the statement of the theorem below) is given in \cite[Theorem 3.8]{row2012technique}.
\begin{proposition}\label{prop:cutbound}
Let $G=(V, E)$ be a graph with a vertex cut $C$. 
Let $V_1$ be the vertex set of one of the components of $G-C$, and let $V_2=V\setminus (V_1\cup C)$. Let $G_1=G[V_1\cup C]$ and $G_2=G[V_2\cup C]$. We have that
\begin{equation*}
    Z(G_1)+Z(G_2)-|C|\leq Z(G)\leq Z(G_1)+Z(G_2)+|C|.
\end{equation*}
\end{proposition}
\begin{proof}
We first prove the upper bound. Let $A_1$ and $A_2$ be optimal forcing arc sets in $G_1$ and $G_2$, respectively. That is, $|A_1|=|V_1|+|C|-Z(G_1)$ and $ |A_2|=|V_2|+|C|-Z(G_2) $. Remove any arc from $A_1$ and $A_2$ with a tail in $C$. By Lemma~\ref{lem:subset}, the remaining arc sets $A_1'$ and $A_2'$ are still forcing arc sets. At most $|C|$ arcs are removed from each of $A_1$ and $A_2$, since $A_1$ and $A_2$ form collections of dipaths. Therefore, $|A_1'|\geq |A_1|-|C|$ and $|A_2'|\geq |A_2|-|C|$.

By construction, all vertices in $C$ are sources in $A_1'$ and $A_2'$. Applying Proposition~\ref{prop:cut}, we obtain that $A=A_1'\cup \overleftarrow{A_2'}$ is a forcing arc set of $G$. 
Now, 
\begin{eqnarray*}
    Z(G)&= &  |V|-|A|\\
    &=&|V|-(|A_1'|+|A_2'|)\\
    &\leq & |V|-(|A_1|+|A_2|-2|C|)\\
    & = & (|V_1|+|V_2|+|C|)-(|V_1|-Z(G_1)+|V_2|-Z(G_2))\\
    &=& Z(G_1)+Z(G_2)+|C|.
\end{eqnarray*}

We then prove the lower bound. Let $A$ be an optimal forcing arc set of $G$, so $|A|=|V|-Z(G)$. Let $A_1$ and $A_2$ be the subsets of $A$ restricted to $G_1$ and $G_2$, respectively. (Note that $A_1$ and $A_2$ are not necessarily disjoint.)
Any subset of a forcing arc set is again a forcing arc set, so  $A_1$ and $A_2$ are forcing arc sets in $G_1$ and $G_2$, respectively. By Corollary~\ref{cor:sizeBA}, $|A_1|\leq |V_1|+| C|-Z(G_1)$ and $|A_2|\leq |V_2|+|C|-Z(G_2)$. Now,
\begin{eqnarray*}
    Z(G)&= &|V|-|A|\\
    & \geq & |V_1|+|V_2|+|C|-|A_1|-|A_2|\\
    &\geq & Z(G_1)+Z(G_2)-|C|.
\end{eqnarray*}
This completes the proof.
\end{proof}

\section{The approximation algorithm}\label{sec:approxalg}
We introduce our approximation algorithm for computing a zero forcing set in Algorithm~\ref{algorithm}. Before that, we will first review a few graph parameters that motivate our algorithm. 
A \emph{fort} in a graph $G=(V, E)$ is a set $F\subset V$ so that no vertex in $V\setminus F$ has exactly one neighbour in $G$. Thus, even if all vertices in $V\setminus F$ are blue, $F$ will remain white. Any zero forcing set must contain at least one vertex in each fort. 
Thus, $Z(G)$ is bounded below by the size of the largest \emph{fort packing} - a collection of mutually disjoint forts. 
Formally, the {\em fort number} is the largest size of a fort packing of $G$: 
$$\ft(G) = \max\{|\calF|: \calF \text{ is a fort packing}\}.$$
\begin{theorem}[\cite{fast2018effects}] 
Every zero forcing set intersects every fort; that is, $\ft(G)\leq Z(G)$. 
\end{theorem}

Our algorithm assumes that a {\sl path decomposition} of the graph is given. 
First defined in \cite{robertson1986graph}, a path decomposition is a collection of subsets of vertices, often called {\sl bags}, with special properties. 
We will, in fact, use a special type of path decomposition, which we now define formally.

\begin{definition}
    A \emph{nice path decomposition} of a graph $G=(V,E)$ is a collection of vertex subsets $X_0,X_1,\dots X_k, X_{k+1}\subset V$ so that
    \begin{enumerate}[label=(D\arabic*)]
        \item\label{cond:pd1} For each edge $uv\in E$, there is some $1\leq i\leq k$ so that $\{u,v\}\subseteq X_i$.
        \item\label{cond:pd2} For each $v\in V$ and $1\leq i<j\leq k$, if $v\in X_i\cap X_j$, then $v\in X_s$ for all $i\leq s\leq j$.
        \item\label{cond: spd-start-end} $X_0=X_{k+1}=\emptyset$ and all other sets $X_i$, $1\leq i\leq k$ are non-empty.
        \item\label{cond: spd-set difference} For all $0\leq i\leq k$, either $X_{i+1}=X_i\cup \{v\}$ for some $v\in V\setminus X_i$, or $X_{i+1}=X_i\setminus \{v\}$ for some $v\in X_i$.
    \end{enumerate}
The definition of an ordinary path decomposition requires only \ref{cond:pd1} and \ref{cond:pd2}.
The {\em width} of a path decomposition is the integer $\max\{|X_i| - 1:1\leq i\leq k\}$. 
The {\em pathwidth} $w = \pw(G)$ of a graph $G$ is the least integer $w$ such that there is a path decomposition of $G$ of width $w$. 
A path decomposition is {\em optimal} if its width is $\pw(G)$.
We call a set $X_i$ an \emph{end set} if either $\bigcup_{j=1}^{i-1}X_j\subseteq X_i$ or $\bigcup_{j=i+1}^k X_j\subseteq X_i$.
\end{definition}

We will write a nice path decomposition $X_0,X_1,\dots, X_{k+1}$ as $\X_k$ for simplicity, where we assume the inclusion of $X_0= X_{k+1} = \emptyset$, and $k$ refers to the number of non-empty bags.

Nice path decompositions were first introduced in \cite{1994pathwidth} as a special type of {\sl nice tree decomposition}. 
Since path decompositions are also tree decompositions, their results apply here too. 
The following two lemmas, taken from \cite{1994pathwidth}, give an upper bound on the number of bags of a nice path decomposition and the complexity of transforming from a path decomposition into a nice path decomposition.
\begin{lemma}
    [{\cite[Lemma 13.1.2]{1994pathwidth}}] 
    Every graph $G$ of order $n$ and pathwidth $w$ has a nice path decomposition of width $w$. 
    Furthermore, there exists a nice path decomposition $\X_k$ with $k\leq 4|V|$.
\end{lemma}
\begin{lemma}[{\cite[Lemma 13.1.3]{1994pathwidth}}]\label{lem:smooth_pw}
    Let $G$ be a graph of order $n$. For constant $w$, given a path decomposition of a graph $G$ of width $w$ with at most $k\leq 4n$ bags, one can find a nice path decomposition of $G$ of width $w$ and with at most $\ell\leq 4n$ bags in $O(n)$ time.
\end{lemma}

Our algorithm will simultaneously construct a zero forcing set $S$ and a fort packing $\F$ of a graph $G$.
The algorithm uses a nice path decomposition of the graph as its input. 
Let $w$ denote the width of the path decomposition; an analysis of the algorithm will show that $|S|\leq (w+1)|\F|$. 
Thus, if we assume that the path decomposition is optimal, that is, $w=\pw(G)$, then $|S|\leq (\pw(G)+1)|\F|$.

Our algorithm is based on repeated application of Proposition \ref{prop:cut}. The application depends on the following lemma.
The lemma can be considered as a corollary of Statement (2.4) in \cite{robertson1986graph} by the fact that a nice path decomposition is a tree decomposition. 

\begin{lemma}[\cite{robertson1986graph}]\label{lemma:bag-cut}
Let $\X_k$ be a nice path decomposition. For each $1\leq  t\leq k$, either $X_t$ is a vertex cut in $G$, or it is an end set. 
\end{lemma}

We first give a brief intuition of the algorithm. 
Consider a graph $G=(V, E)$ and assume that a nice path decomposition $\X_k$
of $G$ is given.  For $0\leq i\leq j\leq k+1$, let 
\begin{equation*}
    X_i^{j}=\bigcup_{\tau=i}^{j}\,X_\tau\quad\text{ and }\quad G(i,j)=G[X_i^{j}].
\end{equation*} 
The algorithm iteratively builds a fort packing and a zero forcing set for $G(0,t)$, where $0\leq t\leq k+1$. This zero forcing set will contain $X_t$. We will also maintain the corresponding forcing arc set $A$; each vertex in $X_t$ will be a source of this set. To extend the zero forcing set and the fort packing, a second index $z>t$ is found so that $X_t\cup X_{z-1}$ is a zero forcing set for $G(t,z)$ but $X_t\cup X_z$ is not. By Lemma~\ref{lemma:bag-cut}, $X_t$ is a vertex cut in $G(0,z)$, and since $X_t$ is part of the zero forcing set for both components, we may apply Proposition~\ref{prop:cut}. (There are some complications when $X_t\cap X_{z-1}\neq\emptyset$, which we ignore here but will address in the formal proof.) 

Precisely, we reverse the arc set $A$ in $G(0,t)$ and take as our new $S$ the set of sources of the reversed arc set, augmented by adding $X_{z-1}$. The fort packing $\F$ is augmented by adding the set of white vertices remaining when the zero forcing process is performed on $G(t,z)$ with the initial set $X_t\cup X_z$. Thus, at most $w+1$ vertices are added to $S$, and one fort is added to $\F$ in each iteration.

We define the notations that we will use below and state our algorithm in Algorithm~\ref{algorithm}.
Given a forcing arc set $A$, let $\sources(A, V)$ and $\sinks(A, V)$ be the set of sources and sinks of $H=(V, A)$, respectively. From Proposition~\ref{prop:reverse}, we know that both $\sources(A, V)$ and $\sinks(A, V)$ are zero forcing sets. 

Given a set $S\subseteq V$, let $\white(S,G)$ be the set of vertices remaining white after the zero forcing process has been executed on $G$ with the initial blue set $S$; that is, $S$ is a zero forcing set if and only if $\white(S,G)=\emptyset$. 
Recall that the derived colouring of an initial blue set is unique; therefore $\white(S, G)$ is well-defined.
At the end of the forcing process, no forcing move is possible, so no blue vertex has a unique white neighbour. Since each vertex is either blue or white, it follows that $\white(S,G)$ is a fort. 

Given a zero forcing set $S$, let $\fas(S, G)$ be the forcing arc set corresponding to the zero forcing process on $G$ with the initial blue set $S$.
As noted earlier, while the derived colouring is unique, the specific forcing moves of the zero forcing process may not be. We may enforce uniqueness by choosing a canonical representative forcing arc set of each set $S$. For example, we can stipulate that exactly one vertex is forced in each step. We can then index all vertices and, in each step of the process, choose the forcing move involving the eligible blue vertex of the lowest index. In the following, we will thus assume without loss of generality that $\fas(G,S)$ is uniquely defined.

\begin{figure}[ht]
\begin{mdframed}[style=MyFrame,align=center]

{\flushleft
\begin{algorithm}[{\bf Zero forcing set algorithm}]\label{algorithm}
\ \\ 

\textbf{Input:} A graph $G$, a nice path decomposition $X_0, X_1,\dots, X_k,X_{k+1}$ of $G$. \\
\textbf{Output:} A zero forcing set $S$ and a fort packing $\F$ of $G$. \\
\newcounter{algct}
\setcounter{algct}{1}
\begin{enumerate}[label=(L\thealgct)]
    \item  Set $t\leftarrow 0$, $S\leftarrow\emptyset$, $A\leftarrow\emptyset$, $\F\leftarrow\emptyset$. \stepcounter{algct}
    \item\label{line:back2} Set $z\leftarrow t$, $W\leftarrow \emptyset$.\stepcounter{algct}
    \item\label{line:loop} \emph{While} $W=\emptyset$ and ${z\leq k}$:  
    \begin{enumerate}[label=(L\thealgct\alph*)]
        \item Set $z\leftarrow z+1$;
        \item\label{line:W} Set $W\leftarrow \white(X_t\cup X_z\,,\, G(t,z))$.
    \end{enumerate} \stepcounter{algct}

    \item\label{line:revA} Set $A\leftarrow \overleftarrow{A}$.\stepcounter{algct}
    \item\label{line:newSource} Set $S\leftarrow \sources(A,X_0^t)$.\stepcounter{algct}
        \item\label{line:non-fin} {\em If} $W\neq\emptyset$: 
    \begin{enumerate}[label=(L\thealgct\alph*)]
        \item Add $W$ to $\F$;
        \item\label{line:S} Set $S\leftarrow S\cup X_{z-1}$; 
         \item\label{line:fas} Set $A'\leftarrow {A}\cup\fas(X_{t}\cup X_{z-1}\,,\,G(t,z))$;
        \item\label{line:removal} Set $A\leftarrow \{(u,v)\in A'\,:\, v\not\in X_{z}\}$; 
        \item\label{line:exit} Set $t\leftarrow z$ and go to \ref{line:back2}.
    \end{enumerate}\stepcounter{algct}
    \item\label{line:Wempty} {\em If} $W=\emptyset$: 
    \begin{enumerate}[label=(L\thealgct\alph*)]
        \item Set $A\leftarrow {A}\cup\fas(X_{t}\,,\,G(t,z))$.
    \end{enumerate}\stepcounter{algct}
    \item\label{line: return}  Return $S$ and $\F$.\stepcounter{algct}
\end{enumerate}
\end{algorithm}
}

\end{mdframed}
\end{figure}

Finally, we give a helpful lemma.
\begin{lemma}\label{lemma: s}
Let $G$ be a graph and $\X_k$  a nice path decomposition of $G$. 
Let $0\leq t< z\leq k+1$ be two indices. 
If $X_t\cup X_{z-1}$ is a zero forcing set for $G(t,z-1)$ but $X_t\cup X_{z}$ is not a zero forcing set for $G(t,z)$, then $X_{z}\subset X_{z-1}$. 
\end{lemma}
\begin{proof}
By \ref{cond: spd-set difference}, either $X_z \subset X_{z-1}$ or $X_{z-1}\subset X_{z}$ where their difference is one vertex. 
Suppose by contradiction that $X_{z-1}\subset X_z$, and $X_z\setminus X_{z-1}= \{a\}$. Since $S= X_{z-1}\cup X_t$ is a zero forcing set for $G(t,z-1)$, and $G(t,z-1)=G(t,z)-a$, the set $X_t\cup X_z=S\cup \{ a\}$ is a zero forcing set for $G(t,z)$. This contradicts our assumption.
\end{proof}

Before we give the proof of correctness of our algorithm, we make some remarks. 
The loop \ref{line:loop} ends precisely when either $z=k+1$ and the algorithm is finished, or $z$ is such that $X_t\cup X_{z-1}$ is a zero forcing set for $G(t,z-1)$, but $X_t\cup X_z$ is not a zero forcing set for $G(t,z)$. By Lemma~\ref{lemma: s}, this implies that $X_{z}\subset X_{z-1}.$ 
It follows that $G(t,z)$ is a subgraph of $G(t,z-1)$ and $X_t\cup X_{z-1}$ is also a zero forcing set for $G(t,z)$. Step \ref{line:non-fin} is the application of the construction from Proposition~\ref{prop:cut}. 
Line \ref{line:removal} only applies when $X_t\cap X_{z}\neq \emptyset$. After this line, each vertex in $X_{z}$ will be a source vertex in $A$.

As a final remark, the duality between forcing arc sets and zero forcing sets implies that the algorithm can be implemented without maintaining $S$. In particular, lines \ref{line:S} and \ref{line:newSource} can be omitted, and $S$ can be obtained from one application of line \ref{line:newSource} at the end of the algorithm. We form both zero forcing set $S$ and forcing arc set $A$ explicitly for the clarity of our analysis and exposition.

\begin{theorem}\label{thm: proof of correctness}
Let $S$ be the set of vertices
returned by Algorithm~\ref{algorithm} with a given graph $G$ and a nice path decomposition $\X_k$. 
The set $S$ is a zero forcing set for $G$. 
\end{theorem}
\begin{proof}
Let $z_0=0$, and let $z_1,\dots, z_h$ be the values of the index $z$ every time that the algorithm exits the {\em While} loop on \ref{line:loop}; note that $z_h=k+1$ necessarily. 

We will prove by induction that for any $j$, $1\leq j\leq h$, after finishing all steps in \ref{line:non-fin} in the $j$-th iteration, the set $S$ obtained at \ref{line:S} is a zero forcing set for $G(0, z_j)$ and the arc set $A$ obtained after \ref{line:removal} is a forcing arc set of $G(0,z_j)$ with the initial blue set $S$.
Further, we will show that each vertex of $X_{z_{j}}$ is a source of $A$; that is, $X_{z_j}\subset \sources(A, X_0^{z_j})$.

To be precise, let $S_j$ and $A_j$ denote the sets $S$ and $A$ at the end of the $j$-th iteration. 
The statement to be proved by induction on $j$ is that $A_j$ is a forcing arc set of $G(0,z_j)$, $X_{z_j}\subseteq S_j$, and $S_j=\sources(A_j,X_0^{z_j})$. For the base case when $j=0$, the graph $G(0, 0)$ is the empty graph and, therefore $A_0$ and $S_0$, both empty sets, trivially satisfy the statement.

Inductively, we assume that our statement holds for some $0\leq j<h$. We will then show that the statement holds for $j+1$ by analyzing the $(j+1)$-th iteration. For simplicity, let $t=z_j$ and $z=z_{j+1}$. Let $W$ be as it is at the end of the loop \ref{line:loop} in the $(j+1)$-th iteration. We assume first that $W\neq \emptyset$.
This implies $X_t\cup X_{z-1}$ is a zero forcing set for $G(t, z-1)$ but $X_t\cup X_z$ is not a zero forcing set for $G(t,z)$. 
 By Lemma \ref{lemma: s} we have $X_{z}\subset X_{z-1}$ and thus $G(t,z-1)=G(t,z)$. By the induction hypothesis, $A_j$ is a forcing arc set of $G(0,t)$ so that each vertex in $X_t$ is a source in $A_j$. 

By Lemma~\ref{lemma:bag-cut}, the set $X_t$ is a cut in $G(0,z)$. We can then apply Proposition~\ref{prop:cut}, with $G_1=G(0,t)$ and $G_2=G(t,z)$. By induction $\overleftarrow{A_j}$ is a forcing arc set in $G_1$ and by definition, $\fas(X_t\cup X_{z-1},G(t,z))$ is a forcing arc set in $G_2$. Let  $A_{t,z}=\fas(X_t\cup X_{z-1},G(t,z))$. Then applying Proposition \ref{prop:cut} gives that $\overleftarrow{A_j}\cup A_{t,z}$ is a forcing arc set of $G(0,z)$. Note that at \ref{line:revA}, $A$ is set to $\overleftarrow{A_j}$ and at \ref{line:fas}, the two forcing arc sets are combined into a forcing arc set $A'$. 

This satisfies the first part of the induction hypothesis for $z=z_{j+1}$, but there may be vertices in $X_{z}$ that are not sources of $A'$. Specifically, vertices in $X_t\cap X_z$ that were sinks in $\overleftarrow{A_j}$ and sources in $A_{t,z}$ may no longer be sources after these forcing arc sets are merged. This means that we may not have that $X_z\subset \sources(A', G(0,z))$.

This is remedied in line \ref{line:removal}, where arcs in $A'$ that have their head in $X_t\cap X_{z}$ are removed to create $A=A_{j+1}$; by Lemma \ref{lem:subset}, $A$ is still a forcing arc set of $G(0,z)$. As a result, each vertex in $X_{z}$ will be a source in $A$ after this step. We then have that $X_z\subset \sources(A_{j+1},X_0^z)$, which proves the second part of the induction hypothesis for $j+1$.

Finally, we show that $S=S_{j+1}$ equals the set of sources of $A=A_{j+1}$. To do this, we must show that 
$$
S_{j+1} = \sources(\rev{A_j}, X_0^{t})\cup X_{z-1}=\sources(A_{j+1}, X_0^{z}),
$$
where $S_{j+1}$ and $A_{j+1}$ are obtained algorithmically; in particular, $S_{j+1}$ is obtained after the steps~\ref{line:revA}, \ref{line:newSource}, and \ref{line:S}. 
After \ref{line:revA} and \ref{line:fas} we have that 
\begin{equation}\label{eqn:A'}
A' = \rev{A_j}\cup A_{t,z}
\end{equation}
By definition, $\sources(A_{t,z},X_t^z) =X_t\cup X_{z-1}$.
Merging forcing arc sets cannot create new sources, and thus
$$
\sources(A', X_0^z) \subset \sources(\rev{A_j}, X_0^t)\cup X_{z-1}\cup X_t.
$$
Let $v\in X_t\setminus X_{z-1}$ and suppose $v\in \sources (A_{j+1},X_0^z)$. 
Since the step~\ref{line:removal} only can create new sources in $X_{z-1}$, we also have that $v\in \sources(A',X_0^z)$. 
By construction as in~\eqref{eqn:A'}, $v$ is a source of  $A_{t,z}$. 
If $v$ is still a source in $A'$, then there can be no arc in $\rev{A_j}$ with head $v$. 
Thus, $v\in \sources(\rev{A_j},X_0^t)$. It follows that
$$
\sources(A_{j+1}, X_0^z) \subset \sources(\rev{A_j}, X_0^t)\cup X_{z-1}.
$$
We next show the reverse inclusion. We showed earlier that $X_{z-1}\subseteq \sources(A_{j+1},X_0^z)$. 
Consider $u\in \sources(\rev{A_j},X_0^t)$. 
If $u\in \sources(A',X_0^z)\subset \sources(A_{j+1},X_0^z)$ then we are done.  Assume then that $u\not\in \sources(A',X_0^z)$.
Since $u$ is  a source in $\rev{A_j}$
but not a source in $A'=\rev{A_j}\cup A_{t,z}$, there must be an in-going arc $(w,u)$ in $A_{t,z}$. 
So $u$ is not a source in $A_{t,z}$ and therefore not in the forcing set $X_t\cup X_{z-1}$, so $u\in X_t^z\setminus (X_t\cup X_{z-1})$. Since $X_t$ is a cut set in $G(0,z)$, this contradicts that $u\in X_0^t$. 
This proves that 
$$
\sources(A_{j+1}, X_0^z) = \sources(\rev{A_j}, X_0^t)\cup X_{z-1},
$$
as required. 

This completes the proof of the inductive step for the case where $W\neq \emptyset$.
We now consider the case where 
$W = \emptyset$ at the end of loop \ref{line:loop} in the $j$-th iteration. In that case, we must have that $z=z_{j+1}=k+1$ and thus $G(0,z)=G$ and $X_z=\emptyset$. By the induction hypothesis, $X_t\cup X_z=X_t$ is a zero forcing set for $G(t,k+1)$. Using Proposition~\ref{prop:cut}, we immediately obtain that $A_{j+1}=A$ as defined in \ref{line:revA} and \ref{line:Wempty} is a forcing arc set, and $S_{j+1}$ as defined in \ref{line:newSource} is the corresponding zero forcing set.  Therefore, $S$ is a zero forcing set and $A$ is a forcing arc set of $G$ after this iteration as well. 
This completes the inductive proof. 
\end{proof}
\begin{theorem}
    Let $\mathcal{F}$ be the set of forts that Algorithm~\ref{algorithm} returns with a given graph $G$ and its path decomposition $\X_k$. 
    The set $\mathcal{F}$ is a fort packing of $G$; that is, all forts are mutually disjoint. 
\end{theorem}
\begin{proof}
    Following the same setup as in Theorem~\ref{thm: proof of correctness}, we may index every fort $W_j\in\calF$ to be the fort that is added to $\calF$ at \ref{line:W} of the $j$-th iteration of the algorithm, for $1\leq j\leq h$. 
    {We show by induction that after the $j$-th iteration, $W_1,\dots ,W_{j}$ is a collection of disjoint forts contained in $X_0^{z_j}\setminus X_{z_j}$.} This statement is trivially satisfied after initialization at step~(L1). Fix $j$, and assume the statement is true after the $(j-1)$-th iteration.
    We will use the notation that, in the $j$-th iteration, $t=z_{j-1}$, $z=z_j$, and $W_j= \white(X_{t}\cup X_{z}\,,\, G(t,z))$. 

    We then have that $W_j\subset X_{t}^{z}\sm (X_t \cup X_z)$. 
    By Lemma~\ref{lemma:bag-cut}, the set $X_{t}$ separates $X_{t}^{z}\sm X_{t}$ from $X_0^{t}\sm X_{t}$. 
    By induction, all other $W_i\in\F$, for $1\leq i< j$, are subsets of $X_0^t\sm X_t$ and thus disjoint from $W_j$.
\end{proof}
\begin{theorem}\label{thm:ratio}
Let $G$ be a graph and $\X_k$ be a nice path decomposition of width $w$.
Let $S$ and $\F$ be the zero forcing set and fort packing returned by Algorithm~\ref{algorithm}, given $G$ and $\X_k$. 
The set $S$ is of order at most $(w+1)|\F|$.
\end{theorem}
\begin{proof}
Fix a graph $G$ and let $S$ and $\calF$ be the two sets returned by the algorithm.
    Again, we invoke the induction setup in the proof of Theorem~\ref{thm: proof of correctness} and notice that the index $h$ represents the number of iterations. In each iteration except for possibly the last one, when $W=\emptyset$ and $z_h = z=k+1$, exactly one fort is added to $\F$ at \ref{line:W}. It follows that  $|\calF|\leq h$. 
    We proceed by induction to show that $|S_j|\leq (w+1)j$, for all $0\leq j\leq h$. 
    When $j = 0$, we have $X_0^0 = \emptyset$ and, therefore, $|S_j|\leq (w+1)\cdot0 = 0$ indeed. 
    
    Inductively, fix $0\leq j < h$ and assume that $|S_j|\leq (w+1)j$; 
    we will then show that $|S_{j+1}|\leq (w+1)(j+1)$.
    It follows from the proof of Theorem \ref{thm: proof of correctness} that $S_j=\sources(A_j, X_0^{z_j})$. Therefore, after \ref{line:revA}, $S=\sinks(A_j, X_0^{z_j})$, and thus 
    $$|S|=|\sinks(A_j, X_0^{z_j})|=|\sources(A_j, X_0^{z_j})|=|S_j|.$$ 
    At~\ref{line:S}, $X_{z_{j+1}-1}$ is added to $S$. We have that
\begin{align*}
    |S_{j+1}|  
        \leq & |S_j| + |X_{z_{j+1}-1}| \\
        \leq & (w+1)j + (w+1) \\ 
        = &  (w+1)(j+1),
\end{align*}
which was what we wanted. 
If we are in the last iteration and $W=\emptyset$, then $S_{j+1}=\sinks(A_j,X_0^{z_j})$ and thus $|S_{j+1}|=|S_j|$, while we add no forts to $\calF$. Therefore, we again have that the zero forcing set satisfies $|S_{j+1}|\leq (w+1)|\calF|.$ 
\end{proof}
\begin{corollary}\label{cor:factor}
    Algorithm~\ref{algorithm} is a $(\pw(G)+1)$-approximation algorithm. 
\end{corollary}
\begin{corollary}\label{cor:bound}
    For any graph $G$, $\ft(G)\leq Z(G)\leq(\pw(G)+1)\ft(G) $. 
\end{corollary}
\begin{proof}[Proof of Corollary~\ref{cor:factor} and~\ref{cor:bound}]
    Suppose that the given nice path decomposition to Algorithm~\ref{algorithm} is optimal; that is, it is of width $w = \pw(G)$.
    Let $S$ and $\F$ be the zero forcing set and fort packing of $G$ returned by the algorithm. 
    Since finding the fort number is a maximization problem, we have $|\F|\leq \ft(G)$. 
    Applying Theorem~\ref{thm:ratio}, let $\ft=\ft(G)$ and $\pw=\pw(G)$, we have 
    $$|\F|\leq \ft\leq Z(G)\leq |S|\leq (\pw+1)|\F|\leq (\pw+1)\ft\leq (\pw+1)Z(G);$$
    this concludes the inequalities. 
\end{proof}


Finally, we analyze the running time complexity of Algorithm~\ref{algorithm}. 
In \cite{BrimkovFastHicks2019}, the complexity of the zero forcing process was analyzed. In particular, an algorithm was given that computes the  {\em closure} of a set $S\subset V$ of $G = (V, E)$, which is the set of all blue vertices after applying all possible forcing. In our algorithm, we need to compute the set of vertices that remains white; clearly, this can be done with the same algorithm. 
\begin{proposition}[\cite{BrimkovFastHicks2019}]
Let $G=(V, E)$ be a graph of order $n$ and size $m$, and let $S \subset V$. The set $\white(S,G)$ can be found in $O(m+n)$ time.
\end{proposition}

The most time-consuming step in our algorithm is \ref{line:loop}, in particular since it repeatedly executes \ref{line:W}. 
The step~\ref{line:removal} has running time at most $|X_z|\leq pw(G)+1\leq n$. All other steps involve merging sets or reversing forcing arcs; these steps can be done in $O(n)$ time if implemented with proper data structures. Therefore, each iteration has $O(n+m)$ running time. 
The number of iterations is at most the number of bags in the smooth path decomposition; by Lemma~\ref{lem:smooth_pw}, it is at most $4n$. Moreover, as stated earlier, every path decomposition can be transformed into a smooth path decomposition of the same width in $O(n)$ time. This leads to the following proposition about the complexity of our algorithm. 

We assume our graphs are connected, so $m\geq n-1$.

\begin{proposition}
   Let $G=(V, E)$ be a connected graph of order $n$ and size $m$, and suppose a path decomposition of $G$ of width $w$  is given. A zero forcing set of size at most $(w+1)Z(G)$ can be found in $O(nm)$ time.
\end{proposition}

\section{New proofs of old results}\label{sec:newproofs}

The concept of a forcing arc set leads to a new perspective on the question of finding the zero forcing number of a graph. Here, we will present new proofs of known results using this new perspective. 
In particular, we will consider proper interval graphs and the strong product of two graphs. Our proofs will be constructive in the sense that we will explicitly define a forcing arc set that obtains the bound. 

\begin{proposition}
    Any forcing arc set of a graph $G$ can contain at most one arc from any clique in $G$. 
\end{proposition}
\begin{proof}
Let $C$ be a clique in $G$ and let $A$ be a forcing arc set of $G$.
Assume by contradiction that $A$ contains two arcs $(u,v)$ and $(x,y)$ so that $u,v,x,y\in C$. If $v=x$, then $uvy$ is a dipath and $yu$ is an edge of $G$, so $uvyu$ is a chain twist. Similar results hold when $u=y$. If $v\not=x$ and $y\not=u$, then $uvxyu$ is a chain twist.
\end{proof}

A \emph{clique cover} of a graph $G$ is a collection of cliques $\C$ so that, for every edge $e$ of $G$, there is a clique in $\C$ which contains both endpoints of $e$. The \emph{clique cover number} of a graph $G$, denoted $cc(G)$, is the minimum size of any clique cover of $G$. By the previous proposition, a forcing arc set $A$ of $G$ can contain at most one edge from each clique in a clique cover, so $|A|\leq|\C|$. This argument and Proposition \ref{prop:sizeBA} lead to the following corollary.

\begin{corollary}\label{corollary: zf cc}
    For any graph $G$ of order $n$, $Z(G)\geq n-cc(G)$. \hfill \qed
\end{corollary}

This result was first proved in~\cite{FALLAT2007558} using linear algebraic techniques. 
In \cite[Theorem~9]{Huang2010}, it was shown that equality holds for Corollary~\ref{corollary: zf cc} when $G$ is a unit interval graph. An interval graph is a graph whose vertex set can be mapped to a set of connected intervals on the real line with two vertices adjacent if and only if the intersection of their corresponding intervals is nonempty. A unit interval graph is an interval graph where each of the corresponding intervals on the real line has a unit length. A proper interval graph is an interval graph where no interval properly contains another.
Roberts \cite{roberts1969indifference} proved the equivalence of unit interval graphs and proper interval graphs (that is, a graph has a unit interval graph representation if and only if it also has a proper interval graph representation).
Therefore, \cite[Theorem~9]{Huang2010} holds for all proper interval graphs.
We provide a shortened proof using forcing arc sets. 

\begin{theorem}
    For any proper interval graph $G$ of order $n$, $Z(G)=n-cc(G)$.
\end{theorem}
\begin{proof}
Let $G = (V, E)$ be a proper interval graph; it is well-known (for example, \cite{roberts1971compatibility}) that $V$ can be ordered as $v_1,v_2,\dots, v_n$ such that if $v_iv_j\in E$ and $i<j$, then $\{v_i, v_{i+1}, \dots, v_j\}$ forms a clique. Assume such an ordering, and let $\C$ be a minimum clique cover. Without loss of generality, we can assume that $\C$ consists of maximal cliques. 
Let $A$ be the set of arcs $(v_i,v_j)$ where $v_i$ is the vertex with the lowest index in some clique $C\in\C$ and $v_j$ is the vertex with the highest index in $C$. 
    
We first show $A$ satisfies \ref{cond:fas1}. 
Every arc in $A$ uniquely corresponds to a clique in $\C$. Let $(v_i,v_j)$ and $(v_{i'},v_{j'})$ be two distinct arcs in $A$ and let $C$ and $C'$ be the corresponding cliques in $\C$. If $i\leq i'$ and $j'\leq j$, then $C'\subset C$, which contradicts the minimality of $\C$. So $i\leq i'\Rightarrow j<  j'$ and similarly, $i'\leq i\Rightarrow j'< j$.  This implies that $i\neq i'$. Using the contrapositive of the statements we obtain that $j\not= j'$, so the heads and tails of arcs in $A$ are distinct. It follows that each vertex has out-degree at most 1 in $H=(V, A)$.   
Since all arcs in $A$ are oriented from a smaller index to a larger index, $H$ is acyclic. Thus, $A$ satisfies \ref{cond:fas1}.  

By Theorem~\ref{thm:chaintwist}, it remains to show that $A$ does not contain a chain twist; that is \ref{cond:fas3}. 
Suppose $A$ contains a chain twist $v_{j_0}v_{j_1}\dots v_{j_k}v_{j_0}$, indexed so that $j_0$ is of the least value. 
This means $j_0<j_k$ and, by construction, $(v_{j_k},v_{j_0})\not\in A$. By definition, this implies that $(v_{j_0},v_{j_1})\in A$ and $(v_{j_{k-1}},v_{j_k})\in A$.
Let $C\in \C$ be the clique that covers edge $v_{j_0}v_{j_k}$, and let $(v_s,v_t)$ be the corresponding arc in $A$. 
By construction, we have $s\leq j_0< j_k\leq t$. Since $(v_s,v_t)$ and $(v_{j_{k-1}},v_{j_k})$ are in $A$ and $j_k\leq t$, it follows from the previous paragraph that $j_{k-1}<s\leq j_0$. This contradicts the assumption on $j_0$. 
Therefore, $A$ does not contain a chain twist.
    
We now discovered a forcing arc set $A$ with $|A| = cc(G)$. 
By Corollary~\ref{cor:sizeBA}, this implies that $Z(G)\leq n-cc(G)$ and, by Corollary \ref{corollary: zf cc}, we proved the equality. 
\end{proof}

\begin{definition}
Let $G = (V, E)$ and $G' = (V', E')$ be two graphs. 
The {\em strong product} of $G$ and $G'$ is the graph $G\boxtimes G' = (V\times V', E^*)$ where $(u, u')(v, v')\in E^*$ if either
\begin{enumerate}[noitemsep, label=(\alph*)]
    \item $u=v$ and $u'v'\in E'$, 
    \item $uv\in E$ and $u'=v'$, or 
    \item $uv\in E$ and $u'v'\in E'$. 
\end{enumerate}
\end{definition}
Note that we use the same notation for vertices in a product graph and directed edge, but the context to which we refer should be clear. 
\begin{proposition}\label{prop:product}
Let $A$ and $A'$ be forcing arc sets of graphs $G$ and $G'$, respectively. 
Let $G^* = G\boxtimes G'$ and define an arc set $A^*$ of $G^*$ as follows:
$$A^* =\{ ((u,u'),(v,v'))\,:\,(u,v)\in A,\,(u',v')\in A'\}.$$
The arc set $A^*$ is a forcing arc set of $G^*$.
\end{proposition}
\begin{proof}
We first show that $A^*$ satisfies \ref{cond:fas1}.  
Toward contradiction, suppose there exists $(u,u')\in V\times V'$ with out-degree strictly greater than one in $(V\times V', A^*)$. 
There then exist two vertices $(v_1,v_1')$ and  $ (v_2, v_2')$ in  $  V\times V'$ such that 
$ ((u,u'), (v_1,v_1'))$ and $ ((u,u'),(v_2, v_2')) $ are in $ A. $
By construction, this implies that the arcs $(u,v_1)$ and $ (u,v_2)$ are in $ A$; but then $u$ has an out-degree of at least two in $(V, A)$, which contradicts that $A$ is a forcing arc set. 
With a similar proof, we may show that every vertex in $(V\times V', A^*)$ has an in-degree of at most one. 

Next, we show that $A^*$ does not contain a chain twist. 
By contradiction, let $W^* = (u_0, v_0)\dots (u_{k-1}, v_{k-1})(u_0, v_0)$ be a chain twist in $A^*$. 
Let $e_i^*$ denote the edge or arc incident to $(u_i, v_i)$ and $(u_{i+1}, v_{i+1})$.
By the definition of a strong product,  $W_G = u_0\dots u_{k-1}u_0$ is a closed walk in $G$, which may contain repeating vertices. We assume that $W_G$ contains at least two different vertices; if this is not the case, we switch the roles of $G$ and $H$. 

We will {derive a contradiction by showing} that $W_G$ satisfies \eqref{eq:chaintwist}. By Lemma~\ref{lem:closedwalk}, it then follows that $W_G$ contains a chain twist, which contradicts that $A$ is a forcing arc set of $G$. 

If $u_i = u_{i+1}$, then $e_i^*\not\in A^*$ by construction, which then means that $e_{i-1}^*, e_{i+1}^*\in A^*$ by definition.  
That is, if there are consecutive repeating vertices, there are at most two in a row; each such repeating pair is surrounded by two arcs in $A$. Therefore, we may treat any pair of consecutive repeating vertices as one. {Without loss of generality,} we may assume in the following that $W_G$ does not contain any consecutive repeating vertices.  With this assumption, let $e_i$ be the arc or edge in $W_G$ incident to $u_i$ and $u_{i+1}$.

If $e_i\not\in A$, then $e_i^*\not\in A^*$. But since $W^*$ is a chain twist, this implies that $e_{i-1}^*$ and $e_{i+1}^*$ are in $A^*$. By definition, this implies that $e_{i-1}$ and $e_{i+1}$ are in $A$. Thus, \eqref{eq:chaintwist} holds for $W_G$ and $A$ contains a chain twist, which gives the contradiction. 
\end{proof}

Proposition~\ref{prop:product} has a corollary that gives abound on the zero forcing number of the strong product of graphs first shown in~\cite{Huang2010}. 

\begin{corollary}
For any two graphs $G$ and $H$,   
\begin{equation*}
  Z(G\boxtimes H)\leq  n_G\,Z(H)+ n_H\, Z(G)-Z(G)Z(H),   
\end{equation*}
where $n_G$ and $n_H$ are the orders of $G$ and $H$, respectively.
\end{corollary}
\begin{proof} Let $A$ and $A'$ be optimal forcing arc sets of $G$ and $H$, respectively. So $|A|=n_G-Z(G)$ and $|A'|=n_H-Z(G)$. 
We first form a forcing arc set of $A^*$ of $G\boxtimes H$ as in Proposition~\ref{prop:product}.
Any forcing arc set of $G\boxtimes H$ therefore has size at least $|A^*|=|A|\,|A'|$, and thus
\begin{eqnarray*}
Z(G\boxtimes H)&\leq &n_Gn_H-|A^*|\\
 &=& n_Gn_H-(n_G-Z(G))(n_H-Z(H))\\
 &=& n_GZ(H) + n_HZ(G)-Z(G)Z(H).
 \end{eqnarray*}
\end{proof}

\section{Conclusion and Further work}\label{sec:conclusion}

We gave an approximation algorithm for finding a zero forcing set that has an approximation ratio of $\pw(G)+1$. The algorithm returns a zero forcing set $S$ and a fort packing $\F$ of size at most $(\pw(G)+1)|S|$. A python implementation of our algorithm is available at \url{https://github.com/OwenZZY/CJMZ_ZF}. Our algorithm immediately leads to a bound on the ratio between zero forcing number and fort number. It would be interesting to see if this bound can be improved or if a family of graphs attaining the bound can be found.

\begin{problem}
    Is the bound $Z(G)\leq (\pw(G)+1)\ft(G)$ sharp?
\end{problem}

The treewidth of a graph is more commonly used than the pathwidth. Moreover, the pathwidth can be up to a factor $\log(n)$ larger than the treewidth. Adapting the algorithm by starting from a tree decomposition may be possible. 

\begin{problem}
  Can the algorithm be adapted to take a tree decomposition as input? If so, does this lead to a better approximation ratio?  
\end{problem}

The {\em positive semidefinite (PSD) zero forcing} process is a variation of the zero forcing process \cite{barioli2008zero}; it is also known as \emph{treewidth zero forcing} in \cite{mitchell2015zero}. 
Let $B$ be a set of forced vertices in $G$ and let $\mathcal G$ represent a collection of connected components of $G-B$. 
Let $V_i$ denote the vertex set for each component and let $G_i = G[V_i\cup B]$. 
The {\em colour change rule for PSD zero forcing} states that a blue vertex $u$ can force a white vertex $v\in V_i$ if $v$ is the only white neighbour of $u$ in $G_i$. 
The forcing arcs then form a collection of directed trees. 
This leads to our final open problem.

\begin{problem}
    Does the chain twist characterization apply to PSD zero forcing?
    Can we modify our algorithm to apply to treewidth zero forcing and take a tree decomposition as input? 
\end{problem}

\section*{Acknowledgements}

Ben Cameron (grants RGPIN-2022-03697 and DGECR-2022-00446) and Jeanette Janssen (grant RGPIN-2017-05112) gratefully acknowledge research support from NSERC. The work of Zhiyuan Zhang was supported by the Ontario Graduate Scholarship (OGS) Program.

\bibliographystyle{abbrv}
\bibliography{ZeroForcing}

\end{document}